\documentclass[11pt]{article}
\pdfoutput = 1 

\usepackage[english]{babel}

\usepackage[utf8]{inputenc}

\usepackage{geometry}
\usepackage{amsmath}
\numberwithin{equation}{section} 

\usepackage{enumitem}

\usepackage[margin = 1em]{subfig}

\usepackage{graphicx} \graphicspath{{./}{Figures/}}

\usepackage{amsthm}

\newtheoremstyle{italic}
{5pt}
{5pt}
{\itshape}
{}
{}
{}
{.5em}
{\bfseries{\thmname{#1}~\thmnumber{#2}.}\thmnote{~\textnormal{(#3)}}}

\newtheoremstyle{upright}
{5pt}
{5pt}
{\upshape}
{}
{\bfseries}
{}
{.5em}
{\bfseries{\thmname{#1}~\thmnumber{#2}.}\thmnote{~\textnormal{(\textit{#3}\textrm{)}}}}

\theoremstyle{italic}
\newtheorem{theorem}{Theorem}[section]

\newtheorem{proposition}[theorem]{Proposition}

\theoremstyle{upright}
\newtheorem{definition}[theorem]{Definition}

\usepackage{dsfont}
\usepackage{amsfonts}

\usepackage{xparse}

\usepackage{mathtools}

\newcommand\blfootnote[1]{%
  \begingroup%
  \renewcommand\thefootnote{}\footnote{#1}%
  \addtocounter{footnote}{-1}%
  \endgroup%
}%

\newcommand{\unit}{u}

\newcommand{\modu}[1]{\operatorname{mod}(#1)} 

\newcommand{\Nat}{\mathbb{N}} 
\newcommand{\Int}{\mathbb{Z}} 
\newcommand{\Real}{\mathbb{R}} 

\NewDocumentCommand \E { m g }{%
    \IfNoValueTF{#2}
        {\mathbb{E}[#1]}
        {\mathbb{E}_{#1}[#2]} 
    }%

\newcommand{\bld}[1]{\mathbf{#1}} 
\NewDocumentCommand \eb { m g }{%
    \IfNoValueTF{#2}
        {\bld{e}_{#1}}
        {\bld{e}^{(#1)}_{#2}} 
    }%
\NewDocumentCommand \zerob { g }{%
    \IfNoValueTF{#1}
        {\bld{0}}
        {\bld{0}^{(#1)}} 
    }%

\newcommand{\la}{\lambda} 
\newcommand{\al}{\alpha}
\newcommand{\be}{\beta}

\newcommand{\I}{I} 
\newcommand{\M}[1]{M^{(#1)}} 
\newcommand{\Lo}{L} 

\newcommand{\ibpos}[1]{\bld{i}_{\scriptscriptstyle +}(#1)} 
\newcommand{\ipos}[1]{i_{\scriptscriptstyle +}(#1)} 
\newcommand{\ibneg}[1]{\bld{i}_{\scriptscriptstyle -}(#1)} 
\newcommand{\ineg}[1]{i_{\scriptscriptstyle -}(#1)} 

\newcommand{\Func}[1]{\Psi(#1)} 
\newcommand{\funcparam}{\psi} 
\newcommand{\func}[2]{\funcparam_{\scriptscriptstyle #1}(#2)} 

\newcommand{\hb}{\bld{h}} 
\newcommand{\h}[1]{h(#1)} 

\newcommand{\ch}{\eta} 
\newcommand{\cv}{\nu} 

\newcommand{\routing}{\mathcal{R}} 
\newcommand{\servicediscipline}{\mathcal{S}} 
\newcommand{\labar}{\overline{\la}} 

\DeclareRobustCommand{\linesolid}{\raisebox{0.2ex}{\includegraphics{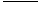}}}
\DeclareRobustCommand{\linedashed}{\raisebox{0.2ex}{\includegraphics{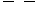}}}
\DeclareRobustCommand{\linedotted}{\raisebox{0.2ex}{\includegraphics{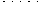}}}

\DeclareRobustCommand{\marktriangle}{\includegraphics{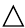}}
\DeclareRobustCommand{\markcircle}{\includegraphics{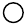}}
\DeclareRobustCommand{\marksquare}{\includegraphics{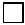}}

\title{Approximate performance analysis of generalized join the shortest queue routing}%
\author{Jori Selen\footnotemark[1] \footnotemark[2], Ivo J.B.F. Adan\footnotemark[1] \footnotemark[2] \footnotemark[3], Stella Kapodistria\footnotemark[2]}%

\begin{document}%

\maketitle%

\renewcommand{\thefootnote}{\fnsymbol{footnote}}%
\footnotetext[1]{Department of Mechanical Engineering, Eindhoven University of Technology}%
\footnotetext[2]{Department of Mathematics and Computer Science, Eindhoven University of Technology}%
\footnotetext[3]{Department of Industrial Engineering \& Innovation Sciences, Eindhoven University of Technology}%
\blfootnote{E-mail address: {\tt j.selen@tue.nl}}%
\renewcommand{\thefootnote}{\arabic{footnote}} \setcounter{footnote}{0}%

\begin{abstract}%
In this paper we propose a highly accurate approximate performance analysis of a heterogeneous server system with a processor sharing (PS) service discipline and a general job-size distribution under a \textit{generalized join the shortest queue} (GJSQ) routing protocol. The GJSQ routing protocol is a natural extension of the well-known join the shortest queue (JSQ) routing policy that takes into account the non-identical service rates in addition to the number of jobs at each server. The performance metrics that are of interest here are the equilibrium distribution and the mean and standard deviation of the number of jobs at each server. We show that the latter metrics are near-insensitive to the job-size distribution using simulation experiments. By applying a \textit{single queue approximation} (SQA) we model each server as a single server queue with a state-dependent arrival process, independent of other servers in the system, and derive the distribution of the number of jobs at the server. These state-dependent arrival rates are intended to capture the inherent correlation between servers in the original system and behave in a rather atypical way.
\end{abstract}%


\section{Introduction}%
\label{sec:introduction}%
%


\subsection{Motivation}%
\label{subsec:motivation}%

This work is motivated by web server farms. Server farms have gained popularity for providing scalable and reliable computing and web services. Most commonly the objective in analyzing such a system lies in the determination of an optimal or near-optimal load balancing routing protocol so as to maximize the performance of the system, see, e.g., \cite{HeterogeneousPSFarms_LoadBalancing_Altman2011,Harchol-Balter2013,HeterogeneousServerFarms_SizeStateAwareRouting_Hyytia2012}, where the performance of interest is usually the mean response time for an arbitrary job. In this paper the objective is to report some interesting properties of the arrival flow to each server and suggest an approximation approach for the GJSQ routing protocol. We consider farms with heterogeneous servers, which is motivated by the different hardware and the wide variety of computing capacities regarding processing power and memory access performance seen in practice in server farms \cite{Ortiz}. We assume that service requests arrive to the system according to a Poisson process. Upon arrival, a front-end dispatcher routes the request to one of the servers. After the request has been routed to the server, we assume that it cannot balk or jokey. All requests routed to a server are sharing the provided service (think of bandwidth, CPU, or RAM). We assume a PS service discipline at each server since it closely approximates the scheduling policies \cite{HB2003,SLA} employed by most commodity operating systems (e.g., Linux CPU time-sharing) and is a popular policy in computing centers (e.g., Cisco Local Director, IBM Network Dispatcher and Microsoft Sharepoint, see \cite{Cardellini2002} for a survey).

In \cite{JSQ_WebServerFarms_Gupta2007} the authors consider a server farm consisting of homogeneous servers, where upon arrival jobs are routed according to the JSQ routing protocol. This protocol in case of homogeneous servers, due to the PS service discipline, is performing near-optimal in terms of the mean response time. However, as indicated by Whitt in \cite{Whitt1986}, the JSQ policy is far from optimal in case of heterogeneous servers. In \cite{Banawan1989} the authors comment on the performance of various systems under different routing protocols and conclude that the shortest expected delay (SED) routing protocol is near-optimal in terms of mean response time. The SED policy is a policy that routes jobs upon arrival to the queue promising the minimum expected delay (which also includes the processing time). In case of exponential job-size distributions, the GJSQ and SED routing protocols are identical and in case of homogeneous servers GJSQ and JSQ are the same. However, in case of general job-size distributions and heterogeneous servers we assume that the only available information are the service rates and the number of jobs at each server, i.e.~we do not keep track of residual processing times.

Due to the complexity and the various challenges that the model at hand presents, we restrict our analysis to the case of two heterogeneous servers with a general job-size distribution under the GJSQ routing protocol. From here onwards we refer to this model as the $M/G(1,s)/2/GJSQ/PS$ system, abbreviated as the GJSQ model, where $G$ is the job-size distribution and 1 and $s$ are the service rates at servers 1 and 2, respectively. The approach described in this paper can be seen as a first stepping stone towards the analysis of heterogeneous server farms with PS servers; a very broad area, full of interesting problems. Moreover, the ideas presented here extend the work of Gupta et al.~\cite{JSQ_WebServerFarms_Gupta2007} on the analysis of the JSQ routing for homogeneous web server farms.


\subsection{Related work}%
\label{subsec:related_work}%

To the best of our knowledge there is no previous mathematical analysis of the GJSQ system. In \cite{SED_Selen2015}, Selen et al. derive the joint equilibrium distribution of the number of jobs at each server in the $M/M(1,s)/2/SED/FCFS$ model. They prove that this distribution can be expressed as an infinite series of product forms using the compensation approach. The benefit of that approach is that it produces, by truncating the series expression, numerical results with an a priori set accuracy level. Unfortunately, the compensation approach is not appropriate (in its current setting) for multiple servers, nor for general job-size distributions. Before \cite{SED_Selen2015}, very little was known regarding the mathematical analysis of the SED policy. In \cite{Lui1995}, the authors suggest two models that act as upper and lower bounds to the SED system. However, they do not provide closed form expressions for the equilibrium distribution of these two bounding models, but only an algorithmic approach based on matrix analytic methods. Furthermore, in \cite{Foschini1978,Laws1992}, the authors show that the SED routing policy is asymptotically optimal in terms of the mean response time and results in complete resource pooling in the heavy traffic limit. This heavy traffic limit result may be used in a similar manner as in \cite{Nelson1989}. However, after a few numerical experiments, we concluded that this approximation in our case results in poor estimates and for this reason we did not proceed in this direction. On the contrary, the approach developed by Gupta et al. \cite{JSQ_WebServerFarms_Gupta2007} on approximating the distribution of the number of jobs at each server, as we show in this paper, is appropriate for the GJSQ setting with heterogeneous servers. More concretely, in \cite{JSQ_WebServerFarms_Gupta2007}, the authors develop the SQA method that accurately determines the distribution of the number of jobs at each server by modeling each queue as an $M_n/M/1/PS$ system with state-dependent arrival rates. These state-dependent arrival rates are referred to as the \textit{conditional arrival rates} and are constructed in such a way that they capture the inherent correlated behavior of the complete server farm.


\subsection{Contributions}%
\label{subsec:contributions}%

We believe that we provide the first approximate analysis of the equilibrium distribution and moments of the number of jobs at each server in the GJSQ system (and by Little's law also the mean response time for an arbitrary job). Moreover, the approximation is highly accurate: we encounter a maximum relative difference between the approximation and simulations of 2.2\%. In deriving these approximations, we provide three key contributions:
\begin{enumerate}[label = \arabic*.]%
\item The mean and standard deviation of the number of jobs at each server and the conditional arrival rates are near-insensitive to the job-size distribution. This allows us to study the more tractable model with an exponential job-size distribution.
\item In case of an exponential job-size distribution, the SQA method yields the same equilibrium distribution for the number of jobs at each server as in the original GJSQ model.
\item For the application of the SQA method we present an approach for the derivation of the conditional arrival rates. In particular, we show that the conditional arrival rates, say $\la_i(n), ~ i = 1,2, ~ n \in \Nat_0$, to server 1 satisfy
    \begin{equation}%
    \la_1(n) \to \rho^{1 + s} \textup{ as } n \to \infty,
    \end{equation}%
    where $\rho$ is the load on the system, see Section~\ref{sec:model_description}, and the conditional rates to server 2 for large $n$ oscillate between $s$ different points. Note that the former result is similar to the result obtained in \cite{JSQ_WebServerFarms_Gupta2007} for the case $s = 1$, however the latter result is very atypical and is discussed in greater detail in Section~\ref{subsec:arrival_rates}.
\end{enumerate}%
%


\subsection{Outline}%
\label{subsec:outline}%

The rest of the paper is organized as follows. In Section~\ref{sec:model_description} we give a detailed model description and formally define and investigate the time-average and conditional arrival rates. Section~\ref{sec:insensitivity_results} is devoted to showing that the performance metrics of interest are near-insensitive to the job-size distribution. We describe the SQA and determine the conditional arrival rates in Section~\ref{sec:single_queue_approximation}. The approximations are evaluated in Section~\ref{sec:evaluating_approximation}. In Section~\ref{sec:conclusion} we present some conclusions.


\section{Model description}%
\label{sec:model_description}%
%


\subsection{Heterogeneous servers}%
\label{subsec:heterogeneous_servers}%

We consider a system of two heterogeneous servers and a single dispatcher. The servers employ a PS service discipline and can have different service rates, i.e.~server 1 has service rate 1 and server 2 has service rate $s$. Jobs arrive to the dispatcher according to a Poisson process with rate $\la$ and are routed immediately to one of the servers. Jobs cannot switch servers after being routed. We detail the routing policy in Section~\ref{subsec:routing_policy}. The size of a job is drawn from a general distribution $G$. Without loss of generality we assume that the mean job size is 1. Note that, for example, the (residual) processing time of a (residual) size $G$ job that runs on server 2 that is currently serving $q_2$ jobs is given by $G q_2/s$.

In what follows we assume that $s$ is a positive integer number. In the general case $s \in \Real_+$ we can bound the corresponding system by two systems with service rates given by the closest two integers to $s$.


\subsection{Routing policy}%
\label{subsec:routing_policy}%

The routing policy employed by the dispatcher is a state-aware policy, i.e.~the dispatcher is aware of the number of jobs at each server just before an arrival instant, $q_1$ and $q_2$, and the service rates. The GJSQ routing policy routes an arriving job to the server with the smallest index $(q_i + 1) / s_i$, where $s_i$ is the service rate at server $i$. In case of a tie, the job is randomly routed to one of the servers. These indexes may be interpreted as an estimate of the expected processing time for the arriving job, made by the dispatcher who is unaware of the job-size distribution and the remaining processing times of the jobs currently in service, and furthermore ignores future arrivals.

Under this routing policy, we define the load on this system as
\begin{equation}%
\rho := \la / (1 + s).
\end{equation}%
Throughout the rest of this paper we assume that $\rho < 1$.

Although not necessarily optimal, GJSQ routing outperforms JSQ routing when servers are non-identical. GJSQ routing attempts to balance the load on the servers by taking into account the different service rates in addition to the information on the current number of jobs at each server. In Figure~\ref{fig:time-average_arrival_rate} we show that the long-term fraction of jobs routed to the two servers is a function of the load $\rho$. In light traffic GJSQ assigns all jobs to the fast server and in heavy traffic the load is divided proportionally according to the service rates. This is in contrast with JSQ routing, which assigns a long-term fraction of the jobs to server 1 that decreases from 1/2 to $1/(1 + s)$ for increasing load $\rho$ (verified through simulation).

\begin{figure}%
\centering%
\includegraphics{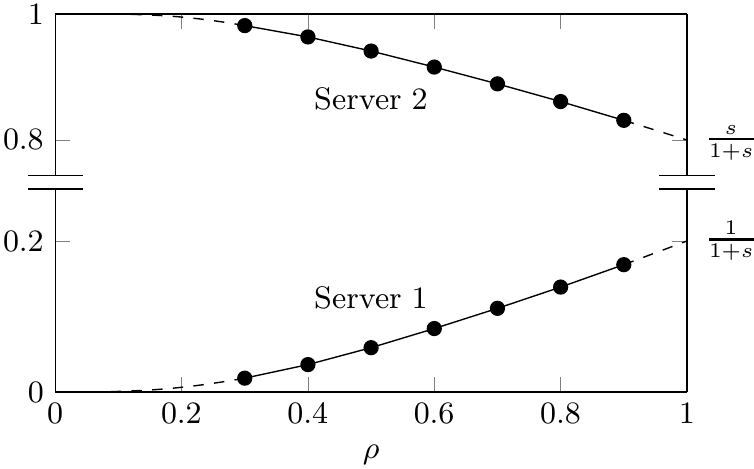}
\caption{Simulated long-term fraction of jobs routed to server $i$ as a function of the load $\rho$, where $s = 4$ and the job-size distribution is exponential. Dashed lines represent expected behavior.}%
\label{fig:time-average_arrival_rate}%
\end{figure}%


\subsection{Arrival rates}%
\label{subsec:arrival_rates}%

We briefly introduce two important concepts related to the (time-average) arrival rates to each server. These concepts will be used throughout the paper.

\begin{definition}\label{def:time-average_arrival_rates}%
In the GJSQ model, the \textit{time-average arrival rate} to server $i$ is defined as
\begin{equation}%
\labar_i := \lim_{t \to \infty} \frac{A_i(t)}{t},
\end{equation}%
where $A_i(t)$ is the number of arrivals at server $i$ during the time interval $[0,t]$.
\end{definition}%

\begin{definition}\label{def:conditional_arrival_rates}%
In the GJSQ model, the \textit{conditional arrival rate} to server $i$, given that server $i$ has $n$ jobs, is defined as
\begin{equation}%
\la_i(n) := \lim_{t \to \infty} \frac{A_{i,n}(t)}{T_{i,n}(t)}, \label{eqn:definition_conditional_arrival_rates}
\end{equation}%
where $A_{i,n}(t)$ is the number of arrivals at server $i$ during the time interval $[0,t]$ that see $n$ jobs at server $i$ on arrival (excluding themselves), and $T_{i,n}(t)$ is the total time spent by server $i$ with $n$ jobs during the time interval $[0,t]$.
\end{definition}%

The two definitions above are related. Assuming it exists, let $\pi_i(n)$ be the equilibrium probability that there are $n$ jobs at server $i$, then $\labar_i = \sum_{n = 0}^\infty \la_i(n) \pi_i(n)$.

Figure~\ref{fig:interesting_conditional_arrival_rates_server_2} depicts the conditional arrival rates to server 2 for varying $s$. Intuitively it makes sense that if a server has many jobs, the other server will probably have few jobs and thus it is less likely that the dispatcher routes the job to that server. However, what we see here is a peculiar repeating pattern that has $s$ different points and does not align with this intuition. We see that if server 2 has a multiple of $s$ jobs (or one less), fewer jobs are routed to server 2. This pattern is difficult to explain, but it is definitely related to the probability that server 1 has a lower index than server 2, given that server 2 currently has $n$ jobs. We expect and indeed verify that this probability also follows such a repeating pattern. Additionally, states in server 2 are somewhat similar if they differ by a multiple of $s$ jobs, which can be derived from the equilibrium distribution in \cite{SED_Selen2015}.

\begin{figure}%
\centering%
\subfloat[$s = 3, ~ \rho = 0.7, ~ \la = 2.8$]{%
\includegraphics[scale = 0.95]{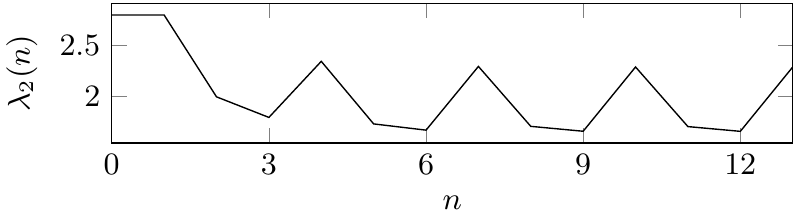}
}%
\subfloat[$s = 4, ~ \rho = 0.7, ~ \la = 3.5$]{%
\includegraphics[scale = 0.95]{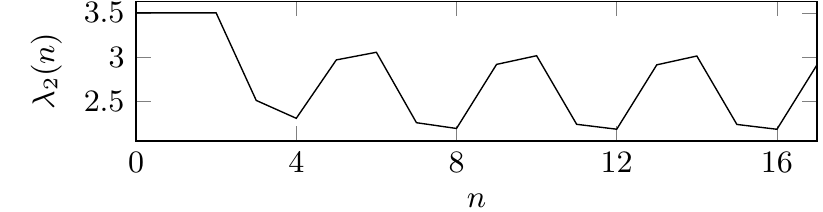}
}%
\caption{The conditional arrival rates to server 2 oscillate between $s$ points.}%
\label{fig:interesting_conditional_arrival_rates_server_2}%
\end{figure}%


\section{Near-insensitivity}%
\label{sec:insensitivity_results}%

In \cite{JSQ_WebServerFarms_Gupta2007} the authors establish a near-sensitivity property in the setting of a homogeneous server farm with JSQ routing. In particular, the first and second moment of the number of jobs at server $i$, $Q_i$, and the conditional arrival rates are near-insensitive to the job-size distribution. The near-insensitivity of these two metrics seems related to the insensitivity of the equilibrium distribution to the job-size distribution in PS servers, see, e.g., \cite[Theorem~4.2]{JSQ_WebServerFarms_Gupta2007}; and the fact that the routing policy only uses the number of jobs at each server when making a decision, as opposed to, e.g., using residual processing times. The GJSQ routing decisions are based on the dynamically changing number of jobs at each server as well as the service rates. Indeed, one expects the near-insensitivity properties to extend also to the case of heterogeneous servers and GJSQ routing. Establishing this near-insensitivity property is important, since it allows us to limit our attention to the more tractable GJSQ system with an exponential job-size distribution.


\subsection{Simulation settings}%
\label{subsec:simulation_settings}%

To support our claims, we simulate the GJSQ model. A simulation consists of $2 \cdot 10^6$ job departures and each simulation is repeated 50 times. Statistics are only computed for departed jobs, i.e.~data of jobs that are still in service at the end of the simulation are discarded. In Table~\ref{tbl:job-size_distributions} we list the four job-size distributions considered in this paper.

\begin{table}%
\centering%
\begin{tabular}{*{5}{l}}%
Name & Distribution & Support & Variance \\
\hline \hline
uni & Uniform & $[0,2]$& $1/3$ \\
exp & Exponential & $[0,\infty)$ & 1 \\
weib & Weibull & $(0,\infty)$ & 5 \\
logn & Log-normal & $(0,\infty)$  & 10
\end{tabular}%
\caption{Job-size distributions used in simulations.}%
\label{tbl:job-size_distributions}%
\end{table}%


\subsection{Near-insensitivity results}%
\label{subsec:near-insensitivity_results}%

In Table~\ref{tbl:simulation_2servers_mean_queue_length} we show simulated statistics on the mean and standard deviation $\sigma(\cdot)$ of $Q_i$ for the GJSQ model with various job-size distributions. For the settings considered in Table~\ref{tbl:simulation_2servers_mean_queue_length}, the mean number of jobs at server $i$ deviates by no more than 3.6\% from the exponential case, while the standard deviation deviates by at most 4.4\%. The largest deviations from the exponential case occur for the log-normal job-size distribution. This is as expected, since this job-size distribution has a variance that is 10 times higher than the variance of the exponential job-size distribution. Although the results are not as strong as those shown in \cite[Figure~3]{JSQ_WebServerFarms_Gupta2007}, we conclude that the more volatile environment of heterogeneous servers and GJSQ routing also has the near-insensitivity property for $\E{Q_i}$ and $\sigma(Q_i)$. Moreover, the performance in terms of the mean response time is also near-insensitive to the job-size distributions by Little's law.

\begin{table}%
\centering%
\footnotesize%
\begin{tabular}{ccc|cccc|cr}%
    &        &        & \multicolumn{4}{c}{Job-size distribution} & \multicolumn{2}{c}{SQA} \\
$s$ & $\rho$ & Metric & uni & exp & weib & logn & Value & Diff. \\
\hline \hline
2 & 0.7 & $\E{Q_1}$     & 0.9139 (0.0030) & 0.9232 (0.0030) & 0.9223 (0.0029) & 0.9361 (0.0038) & 0.9077 & 1.7\% \\
  &     & $\sigma(Q_1)$ & 1.0404 (0.0049) & 1.0505 (0.0050) & 1.0560 (0.0056) & 1.0704 (0.0067) & 1.0462 & 0.4\% \\
  &     & $\E{Q_2}$     & 2.0111 (0.0059) & 2.0289 (0.0061) & 2.0222 (0.0057) & 2.0519 (0.0074) & 2.0329 & $-$0.2\% \\
  &     & $\sigma(Q_2)$ & 2.0302 (0.0099) & 2.0465 (0.0106) & 2.0506 (0.0114) & 2.0813 (0.0141) & 2.0484 & $-$0.1\% \\
\cline{2-9}
  & 0.9 & $\E{Q_1}$     & 3.2244 (0.0336) & 3.2797 (0.0336) & 3.2316 (0.0298) & 3.2396 (0.0266) & 3.2188 & 1.9\% \\
  &     & $\sigma(Q_1)$ & 3.2208 (0.0590) & 3.2716 (0.0723) & 3.2186 (0.0575) & 3.2002 (0.0505) & 3.2161 & 1.7\% \\
  &     & $\E{Q_2}$     & 6.6841 (0.0676) & 6.7915 (0.0674) & 6.6834 (0.0587) & 6.6988 (0.0524) & 6.6424 & 2.2\% \\
  &     & $\sigma(Q_2)$ & 6.4289 (0.1185) & 6.5288 (0.1453) & 6.4186 (0.1141) & 6.3828 (0.1016) & 6.4091 & 1.9\% \\
\hline \hline
4 & 0.7 & $\E{Q_1}$     & 0.4688 (0.0016) & 0.4747 (0.0017) & 0.4705 (0.0017) & 0.4667 (0.0022) & 0.4741 & 0.1\% \\
  &     & $\sigma(Q_1)$ & 0.6685 (0.0024) & 0.6730 (0.0026) & 0.6700 (0.0029) & 0.6652 (0.0031) & 0.6655 & 1.1\% \\
  &     & $\E{Q_2}$     & 2.5386 (0.0063) & 2.5507 (0.0069) & 2.5177 (0.0070) & 2.4997 (0.0067) & 2.5866 & $-$1.4\% \\
  &     & $\sigma(Q_2)$ & 2.5082 (0.0102) & 2.5179 (0.0115) & 2.4936 (0.0133) & 2.4744 (0.0122) & 2.5457 & $-$1.1\% \\
\cline{2-9}
  & 0.9 & $\E{Q_1}$     & 1.8662 (0.0191) & 1.8793 (0.0145) & 1.8830 (0.0196) & 1.9400 (0.0223) & 1.8813 & $-$0.1\% \\
  &     & $\sigma(Q_1)$ & 1.9404 (0.0338) & 1.9539 (0.0314) & 1.9801 (0.0444) & 2.0394 (0.0408) & 1.9566 & $-$0.1\% \\
  &     & $\E{Q_2}$     & 8.2405 (0.0769) & 8.2773 (0.0597) & 8.2631 (0.0783) & 8.4863 (0.0861) & 8.3642 & $-$1.0\% \\
  &     & $\sigma(Q_2)$ & 7.6982 (0.1374) & 7.7507 (0.1264) & 7.8485 (0.1815) & 8.0912 (0.1652) & 7.7692 & $-$0.2\% \\
\end{tabular}%
\caption{Simulated mean and standard deviation of $Q_i$, for the GJSQ system with various $s$, $\rho$ and job-size distributions. Sample standard deviation is shown in parentheses. Last two columns show the value obtained by the SQA and the relative difference with respect to the exponential case.}%
\label{tbl:simulation_2servers_mean_queue_length}%
\end{table}%

Concerning the conditional arrival rates, we see in Figure~\ref{fig:conditional_arrival_rate} that the simulated values for the job-size distributions of Table~\ref{tbl:job-size_distributions} match the results of the algorithm for the exponential case \cite{SED_Selen2015}. Simulated values for states where the sample standard deviation is not too high differ by at most 5\% from the results for the exponential case. So, also the conditional arrival rates are near-insensitive to the job-size distribution.

\begin{figure}%
\centering%
\subfloat[$\rho = 0.7, ~ s = 4, ~ \la = 3.5$]{%
\includegraphics{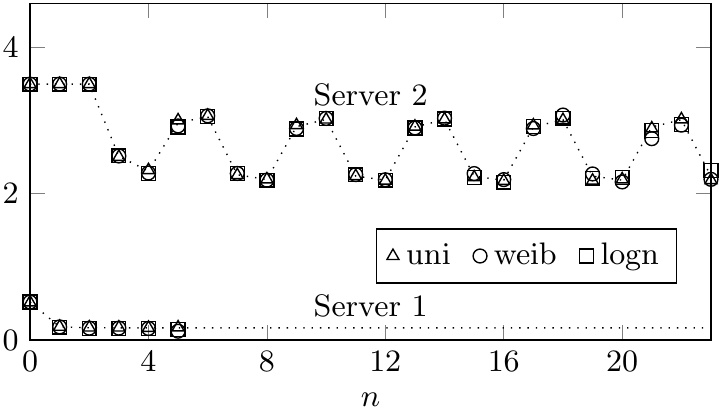}
}%
\subfloat[$\rho = 0.9, ~ s = 4, ~ \la = 4.5$]{%
\includegraphics{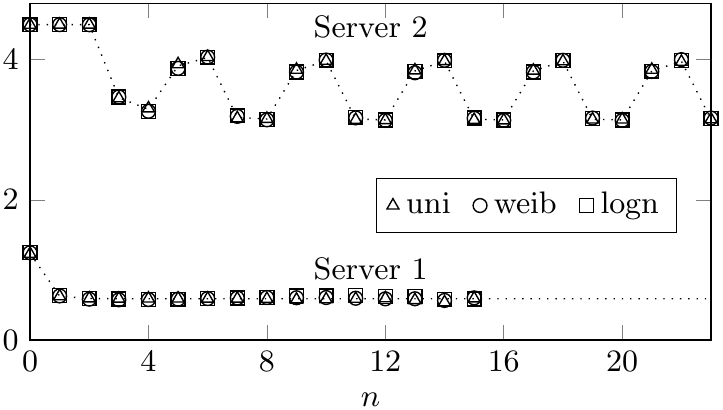}
}%
\caption{Simulated conditional arrival rates in the GJSQ system with various job-size distributions. The dotted curves represent values determined by the algorithm in \cite{SED_Selen2015} for the exponential job-size distribution.}%
\label{fig:conditional_arrival_rate}%
\end{figure}%


\section{Single queue approximation}%
\label{sec:single_queue_approximation}%

We have established near-insensitivity of $\E{Q_i}$, $\sigma(Q_i)$ and the conditional arrival rates to the job-size distributions. Thus, we may limit our attention to systems with an exponential job-size distribution. In this section we derive an approximation for the distribution of the number of jobs at each server using the SQA, which models server $i$ as an $M_n/M_i/1/PS$ queue with state-dependent arrival rates $\la_i(n)$, see also \cite[Section~3]{JSQ_WebServerFarms_Gupta2007}. SQA is exact when the job-size distribution is exponential and the routing belongs to a specific class of routing policies; the following theorem is a version of \cite[Theorem~3.1]{JSQ_WebServerFarms_Gupta2007} that is applicable to the GJSQ model.

\begin{definition}\label{def:stationary_state-aware_routing_policy}%
A \textit{stationary state-aware routing policy} is a time-stationary routing policy that only uses information about the number of jobs at the servers and the service rates at the instant of an arrival. The decisions may be made probabilistically, possibly biased in favor of certain servers.
\end{definition}%

\begin{theorem}\label{thm:SQA_exact_Markovian}
Consider the $M/M(1,s)/2/\routing/\servicediscipline$ queueing model, where $\routing$ is any stationary state-aware routing policy, e.g.~GJSQ, and $\servicediscipline$ is any stationary, size-independent, work-conserving service discipline, e.g.~PS. Consider server $i$ in this model. Then SQA with the exact conditional arrival rates $\la_i(\cdot)$ yields the same equilibrium distribution for the number of jobs at each server as in the original model.
\end{theorem}%

It remains to specify the conditional arrival rates $\la_i(n)$ for both servers. We combine exact limiting results for $n \ge N_i$ and approximation results for $n < N_i$, where $N_1 = 3$ and $N_2 = 2s$. These choices for $N_i$ result in accurate approximations.

We note that Theorem~\ref{thm:SQA_exact_Markovian} implies that in order to determine the conditional arrival rates, we may assume a FCFS service discipline.

In Figure~\ref{fig:interesting_conditional_arrival_rates_server_2} we have seen that the conditional arrival rates $\la_i(n)$ exhibit a repeating pattern from some $n$ and onwards. We rigorously characterize this limiting repeating pattern in the next theorem.

\begin{theorem}\label{thm:limiting_conditional_arrival_rates}%
For the $M/M(1,s)/2/GJSQ/PS$ queueing model with $s \in \Nat$,
\begin{align}%
\la_1^{\textup{lim}} \coloneqq \lim_{n \to \infty} \la_1(n) &= \rho^{1 + s}, \label{eqn:limiting_conditional_arrival_rate_server_1}\\
\la_2^{\textup{lim}}(r) \coloneqq \lim_{n \to \infty} \la_2(sn + r) &= \begin{cases}%
s \frac{A(r + 1)}{A(r)}, & r = 0,1,\ldots,s - 2, \\
s \rho^{1 + s} \frac{A(0)}{A(s - 1)}, & r = s - 1,
\end{cases}\label{eqn:limiting_conditional_arrival_rate_server_2}%
\end{align}%
where
\begin{equation}%
A(r) = \sum_{i = 1}^s \ch_i \frac{\be_i}{\rho^{1 + s} - \be_i} \ipos{\rho^{1 + s},\be_i,r} + h(r) + \frac{\be_{s + 1}}{1 - \be_{s + 1}} \ineg{\rho^{1 + s},\be_{s + 1},r}, \label{eqn:A(r)}
\end{equation}%
and the variables $\be_1,\be_2,\ldots,\be_{s+1}$, $\ch_1,\ch_2,\ldots,\ch_{s}$, and the functions $h(\cdot)$, $\ipos{\cdot}$, $\ineg{\cdot}$ are defined in Appendix~\ref{app_sec:definitions}.
\end{theorem}%

\begin{proof}%
See Appendix~\ref{app_sec:proof_limiting_conditional_arrival_rates}.
\end{proof}%

For the rates $\la_1(n), ~ n < 3$ and $\la_2(n), ~ n < 2s$ we provide approximations that are functions of $s$ and $\rho$. For server 1 we use a multiple linear regression model to fit an approximate function for the conditional arrival rates on data obtained from the algorithm in \cite{SED_Selen2015} for $s = 1,2,3,4$ and $\rho$ from 0.3 to 0.99. Obviously, one can also use conditional arrival rates obtained by simulation for these fitting purposes. We carefully select a set of 5 independent variables for each conditional arrival rate. This leads to the following approximate conditional arrival rates for server 1:
\begin{align}%
\frac{\la_1(0)}{\rho^{1 + s}} &\approx \begin{bmatrix} s\rho & s & \frac{s}{\rho} & 1 & \frac{\rho^2}{s^2} \end{bmatrix} \begin{bmatrix} 0.669 & -1.90 & 1.23 & 1.86 & -0.192 \end{bmatrix}^T, \label{eqn:conditional_arrival_rate_server_1_n_0} \\
\frac{\la_1(1)}{\rho^{1 + s}} &\approx  \begin{bmatrix} s\rho^2 & 1 & \frac{1}{\rho} & \frac{1}{s\rho} & \rho^{1/s} \end{bmatrix} \begin{bmatrix} -0.00856 & 1.37 & -0.0578 & 0.123 & -0.254 \end{bmatrix}^T, \label{eqn:conditional_arrival_rate_server_1_n_1} \\
\frac{\la_1(2)}{\rho^{1 + s}} &\approx 1 + \begin{bmatrix} s\rho & \frac{1}{s\rho} & \frac{\rho}{s^2} & \frac{1}{s^2} & \rho^{1/s} \end{bmatrix} \frac{1}{100} \begin{bmatrix} -0.131 & -0.820 & -6.48 & 10.4 & 0.893 \end{bmatrix}^T \label{eqn:conditional_arrival_rate_server_1_n_2}
\end{align}%
with $\bld{x}^T$ the transpose of a vector $\bld{x}$. For $s = 1$, one should consider $\la_1(\cdot) = \la_2(\cdot)$ and use the approximations presented in \eqref{eqn:conditional_arrival_rate_server_1_n_0}-\eqref{eqn:conditional_arrival_rate_server_1_n_2}.

For server 2, let us note that $\la_2(n) = \la, ~ n = 0,\ldots,s - 2$ due to the GJSQ routing. Using a multiple regression model in this case is more difficult, since the number of states for which we need to obtain a fit increases with $s$. To circumvent a possibly complex fitting procedure, we establish a relation between the conditional arrival rates for the states $n = s - 1,s,\ldots,2s - 1$ and the limiting conditional arrival rates determined in Theorem~2. Namely,
\begin{align}%
\la_2(n) &\approx \Bigl( 1 + \bigl( \frac{1}{s} - \frac{\rho}{2s - 1} \bigr)  \frac{1}{2^{n - (s - 1)}} \Bigr) \la_2^{\mathrm{lim}}(n - s), \label{eqn:conditional_arrival_rate_server_2_n_s-1_until_2s-1}
\end{align}%
where for convenience $\la_2^\textup{lim}(-1) = \la_2^\textup{lim}(s - 1)$. The approximations \eqref{eqn:conditional_arrival_rate_server_1_n_0}-\eqref{eqn:conditional_arrival_rate_server_2_n_s-1_until_2s-1} behave in various limiting regimes as expected:

\begin{proposition}\label{prop:conditional_arrival_rates_limiting_regimes}\hspace*{1em}%
\begin{enumerate}[label = \textup{\arabic*.}]%
\item For $s \to \infty$, we have that $\la_1(n) \downarrow 0$ and $\la_2(n) = \la$ for all $n \in \Nat_0$. No job will join server 1, since the processing times in server 2 are instantaneous.
\item In the light-traffic regime, i.e.~$\rho \downarrow 0$, we find that $\la_1(n) \downarrow 0, ~ n \in \Nat_0$ and $\la_2(n) \downarrow 0, ~ n \ge s - 1$.
\item In the heavy-traffic regime, i.e.~$\rho \uparrow 1$, we establish that $\labar_1/\la = 1/(1 + s)$ and $\labar_2/\la = s/(1 + s)$ which is consistent with the findings in Figure~\ref{fig:time-average_arrival_rate}.
\end{enumerate}%
\end{proposition}%

\begin{proof}%
1. Follows straightforwardly by taking the limit $s \to \infty$ in \eqref{eqn:conditional_arrival_rate_server_1_n_0}-\eqref{eqn:conditional_arrival_rate_server_1_n_2} while taking into account that $\rho = \la/(1 + s)$. Furthermore, observe that $\la_2(n) = \la, ~ n = 0,\ldots,s - 2$, so that $\lim_{s \to \infty} \la_2(n) = \la, ~ n \in \Nat_0$.

\noindent 2. See Appendix~\ref{app_sec:proof_limiting_regimes}.

\noindent 3. From the approximate conditional arrival rates $\la_1(\cdot)$ one can derive (approximate) equilibrium probabilities $\pi_1(\cdot)$. Then, $\labar_1 = \sum_{n = 0}^\infty \la_1(n) \pi_1(n) = \sum_{n = 0}^\infty \pi_1(n+1) = 1 - \pi_1(0)$ by exploiting the balance equations. For $\rho \uparrow 1$ it can be verified that $\pi_1(0) \downarrow 0$, so that $\lim_{\rho \uparrow 1} \labar_1/\la = 1/(1 + s)$. The result for server 2 follows analogously.
\end{proof}%


\section{Evaluating the approximation}%
\label{sec:evaluating_approximation}%

We are now in a position to evaluate the proposed approximations. First, we show that the approximations for the conditional arrival rates follow closely the exact values, which were determined using the algorithm \cite{SED_Selen2015}. Second, we establish that the mean and standard deviation of the number of jobs at each server is also well approximated.

Figure~\ref{fig:comparison_conditional_arrival_rates} compares the conditional arrival rates obtained from the algorithm in \cite{SED_Selen2015} and the approximations derived in the previous section. For the cases considered in the figure, the maximum relative difference of the approximation with respect to the values determined by the algorithm is 1.5\% for $\la_1(\cdot)$ and 4.1\% for $\la_2(\cdot)$. Since both methods consider exponential job-size distributions, the difference is due to the fitting error introduced in the approximations of the conditional arrival rates in Section~\ref{sec:single_queue_approximation} and the truncation error in the algorithm in \cite{SED_Selen2015}. However, since the truncation error has been chosen to be of the order $10^{-5}$, it has little influence.

\begin{figure*}%
\centering
\subfloat[$s = 3$, server 1]{%
\includegraphics{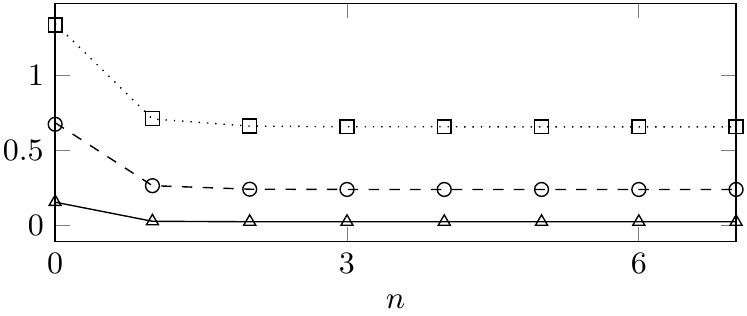}
}%
\subfloat[$s = 3$, server 2]{%
\includegraphics{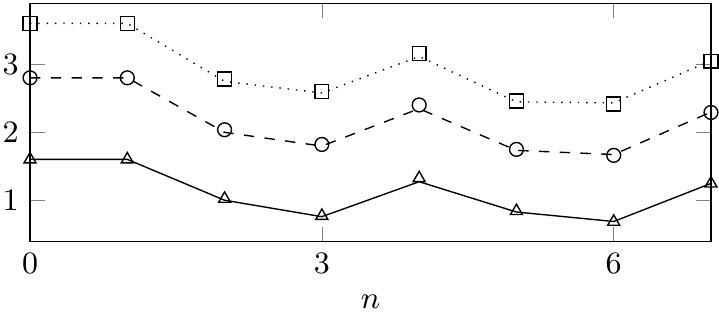}
}%
\\%
\subfloat[$s = 4$, server 1]{%
\includegraphics{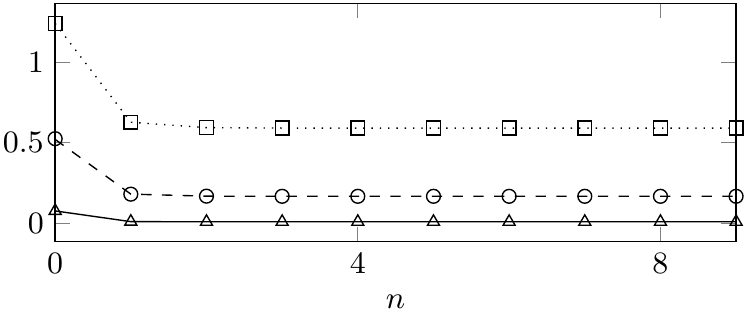}
}%
\subfloat[$s = 4$, server 2]{%
\includegraphics{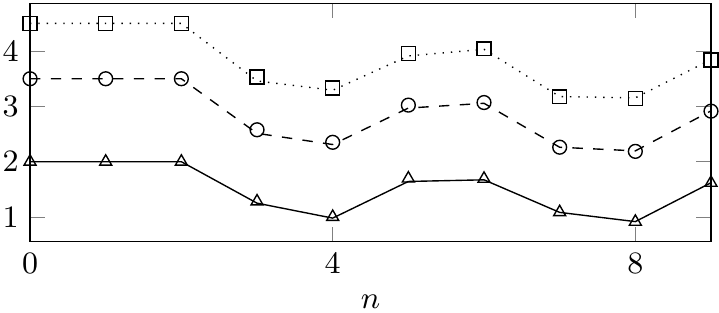}
}%
\caption{Comparison of conditional arrival rates determined by the algorithm in \cite{SED_Selen2015} (lines) and our approximations (marks) for $\rho = 0.4$ (\linesolid,\marktriangle), $\rho = 0.7$ (\linedashed,\markcircle), and $\rho = 0.9$ (\linedotted,\marksquare).}%
\label{fig:comparison_conditional_arrival_rates}
\end{figure*}%

In the two rightmost columns of Table~\ref{tbl:simulation_2servers_mean_queue_length} we provide the mean and standard deviation of the number of jobs at both servers determined using the SQA. We report highly accurate approximations that deviate less than 2.2\% from the case with an exponential job-size distribution for the listed values of $s$ and $\rho$. Although our approximations are not aimed at the case $s = 1$, we report accurate approximations also in this setting with maximum relative differences of the same order as in \cite[Section~6.1]{JSQ_WebServerFarms_Gupta2007}.


\section{Conclusion}%
\label{sec:conclusion}%

In this paper, we provide an approximate performance analysis of a queueing system consisting of two heterogeneous PS servers with service rates 1 and $s \in \Nat$, respectively, a general job-size distribution and GJSQ routing. More concretely, we derived the approximate equilibrium distribution of the number of jobs at each server using the SQA method. In order to apply SQA, we established that the GJSQ system is near-insensitive to the job-size distribution and thus we approximated the system at hand with exponentially distributed job-sizes. We then approximated the conditional arrival rates for the exponential case, by combining exact limiting results for large number of jobs and approximation results, which were obtained using a multiple linear regression model, for small number of jobs. Ultimately, the aforementioned approach resulted in approximations that are highly accurate; we reported a maximum relative difference with respect to exact or simulation results of 4.1\% for the conditional arrival rates and 2.2\% for the mean and standard deviation of the number of jobs at each server.

In this paper we set the groundwork for the analysis of server farms with heterogeneous servers under the GJSQ routing policy by analyzing the case of two servers. Of course, server farms consist of multiple servers so it is in our future plans to extend the analysis presented in this paper to more than two servers. The most difficult aspect of this task would be the derivation of the conditional arrival rates, which possibly has to rely on simulation data, since the approach in \cite{SED_Selen2015} is in its current setting restricted to two servers. In Figure~\ref{fig:conditional_arrival_rates_three_servers} we present an example of the simulated conditional arrival rates in case of three servers with service rates 1, 2 and 5. Note that the structure of the conditional arrival rates is as expected, i.e.~the number of points in the repeating pattern is directly related to the rate of service, but the exact values of these points differ from the values obtained by formulas \eqref{eqn:limiting_conditional_arrival_rate_server_1} and \eqref{eqn:limiting_conditional_arrival_rate_server_2}.

\begin{figure}%
\centering%
\includegraphics{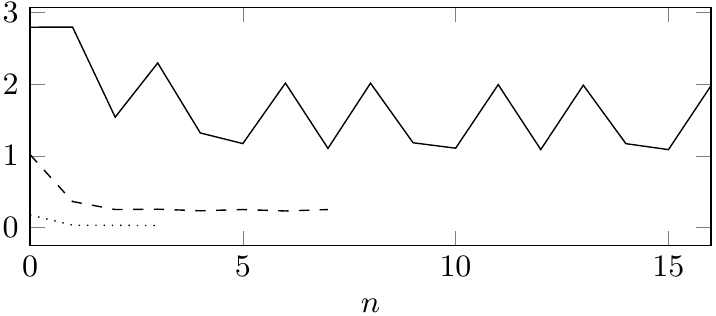}
\caption{Simulated conditional arrival rates for a system with three servers with service rates 1 (\linedotted), 2 (\linedashed), and 5 (\linesolid), with $\rho = 0.7$.}%
\label{fig:conditional_arrival_rates_three_servers}%
\end{figure}%


\subsection*{Acknowledgments}%
\label{subsec:acknowledgements}%

This work was supported by an NWO free competition grant and the NWO Gravitation Project NETWORKS.


\appendix%


\section{Definition of the variables and functions in Theorem~\ref{thm:limiting_conditional_arrival_rates}}%
\label{app_sec:definitions}%

The definitions can be found in \cite[Lemma~5.11]{SED_Selen2015}, but we summarize them here. We denote $\al = \rho^{1 + s}$.

The functions $\ipos{\cdot}$ and $\ineg{\cdot}$ are vectors of size $s$ and have their first element equal to 1, i.e.~$\ipos{\al,\be,0} = \ineg{\al,\be,0} = 1$. Furthermore, the vectors satisfy
\begin{equation}%
\frac{\ipos{\al,\be,r}}{\ipos{\al,\be,0}} = \Bigl( \frac{\al \be (1 + s)(\rho + 1) - \be^2 (1 + s) \rho - \al^2}{\al \be s} \Bigr)^r, \quad r = 0,1,\ldots,s - 1 \label{eqn:inner_pos_eigenvector}
\end{equation}%
and
\begin{equation}%
\frac{\ineg{\al,\be,r}}{\ineg{\al,\be,0}} =  \frac{\Func{\al,\be,\func{-}{\be}} \func{+}{\be}^r - \Func{\al,\be,\func{+}{\be}} \func{-}{\be}^r}{\Func{\al,\be,\func{-}{\be}} - \Func{\al,\be,\func{+}{\be}}}, \quad r = 0,1,\ldots,s - 1
\end{equation}%
with
\begin{equation}%
\func{\pm}{\be} = \frac{(1 + s)(\rho + 1) - \be \pm \sqrt{(\be - (1 + s)(\rho + 1))^2 - 4s(1 + s)\rho}}{2s}
\end{equation}%
and
\begin{equation}%
\Func{\al,\be,\funcparam} = \be - (1 + s)(\rho + 1) + s \funcparam + \frac{\be}{\al}(1 + s)\rho \funcparam^{s-1} = (1 + s) \rho \funcparam^{-1} \bigl( \frac{\be}{\al} \funcparam^s - 1 \bigr).
\end{equation}%

The variables $\be_1,\be_2,\ldots,\be_s$ are the $s$ roots that satisfy $|\be_i| < |\al|, ~ i = 1,2,\ldots,s$ of the equation
\begin{equation}%
\bigl( \al\be(1 + s)(\rho + 1) - \be^2(1 + s)\rho - \al^2 \bigr)^s - \be (\al\be s)^s = 0, \label{eqn:determinant_inner_pos}
\end{equation}%
and $\be_{s + 1}$ with $|\be_{s + 1}| < |\al|$ is the single root of
\begin{equation}%
\al^2 s^s + \be^2 ((1 + s)\rho)^s - \al \be s^s ( \func{+}{\be}^s + \func{-}{\be}^s ) = 0. \label{eqn:determinant_inner_neg}
\end{equation}%

The vector $\hb = (\h{0},\h{1},\ldots,\h{s - 1})$ is given by
\begin{equation}%
\hb = \al \bigl( \frac{1}{2}(1 + s)\rho \M{0,s - 1} + \al \I \bigr)^{-1} \sum_{i = 1}^s \ch_i \ibpos{\al,\be_i},
\end{equation}%
where the coefficients $\ch_1,\ch_2,\ldots,\ch_s$ satisfy
\begin{align}%
&\sum_{i = 1}^s  \ch_i \Bigl( \be_i \bigl( (1 + s)\rho \I + \al s\M{s - 1,0} \bigr) \notag \\
&\hspace{1.2cm} + \al^2 \bigl( -(1 + s)(\rho + 1)\I + s\Lo^T + (1 + s)\rho \Lo \bigr) \bigl( \frac{1}{2}(1 + s)\rho \M{0,s - 1} + \al \I \bigr)^{-1} \Bigr) \ibpos{\al,\be_i} \notag \\
&= - \be_{s + 1} \bigl( (1 + s)\rho\M{0,s - 1} + \al \I \bigr) \ibneg{\al,\be_{s + 1}},
\end{align}%
where $\I$ is the $s \times s$ identity matrix, $\M{x,y}$ is the $s \times s$ binary matrix with element $(x,y)$ equal to one and zeros elsewhere, and $\Lo$ is an $s \times s$ subdiagonal matrix with elements $(x,x - 1), ~ x = 1,2,\ldots,s - 1$ equal to one and zeros elsewhere. For consistency with indexing of all vectors, the indexing of a matrix starts at 0.


\section{Proof of Theorem~\ref{thm:limiting_conditional_arrival_rates}}%
\label{app_sec:proof_limiting_conditional_arrival_rates}%

The proof is based on the exact results of the related $M/M(1,s)/2/SED/FCFS$ system, with $s \in \Nat$, presented in \cite{SED_Selen2015}. Although we obtain similar results for the limiting conditional arrival rates for server 1 as in \cite{JSQ_WebServerFarms_Gupta2007}, we use here a completely different approach in deriving the limits.

In \cite{SED_Selen2015}, the state space $\{ (q_1,q_2) \mid (q_1,q_2) \in \Nat_0^2\}$ of the Markov process is transformed to the state space $\{ (m,n,r) \mid m \in \Nat_0, ~ n \in \Int, ~ r = 0,1,\ldots,s - 1 \}$ where $m = \min(q_1,\lfloor \frac{q_2}{s} \rfloor)$, $n = \lfloor \frac{q_2}{s} \rfloor - q_1$ and $r = \modu{q_2,s}$. Let us denote the equilibrium probabilities for the three-dimensional state space as $p(m,n,r)$. The equilibrium probability $p(m,n,r)$ has a series expression, i.e.~$p(m,n,r) = \sum_{l = 0}^\infty x(l,m,n,r)$, namely, for $m \ge 0, ~ n \ge 1$,
\begin{align}%
p(m,n,r) &= C \sum_{l = 0}^\infty \sum_{i = 1}^{(s + 1)^l} \sum_{j = 1}^s \be_{l,d(i) + j}^n \bigl( \ch_{l,d(i) + j} \al_{l,i}^m  \notag \\
&\hspace{4.5cm}+ \cv_{l + 1,d(i) + j} \al_{l + 1,d(i) + j}^m \bigr) \ipos{\al_{l,i},\be_{l,d(i) + j},r}. \label{eqn:equilibrium_distribution_3d_pos}
\intertext{For $m \ge 0$,}
p(m,0,r) &= C \sum_{l = 0}^\infty \sum_{i = 1}^{(s + 1)^l} \al_{l,i}^m h_{l,i}(r). \label{eqn:equilibrium_distribution_3d_hor}
\intertext{For $m \ge 0, ~ n \le -1$,}
p(m,n,r) &= C \sum_{l = 0}^\infty \sum_{i = 1}^{(s + 1)^l} \ch_{l,i(s + 1)} \al_{l,i}^m \be_{l,i(s + 1)}^{-n} \ineg{\al_{l,i},\be_{l,i(s + 1)},r} \notag \\
&\quad + C \sum_{l = 0}^\infty \sum_{i = 1}^{(s + 1)^l} \cv_{l+1,i(s + 1)} \al_{l + 1,i(s + 1)}^m \be_{l,i(s + 1)}^{-n} \ineg{\al_{l + 1,i(s + 1)},\be_{l,i(s + 1)},r}. \label{eqn:equilibrium_distribution_3d_neg}
\end{align}%
For the exact interpretation of each variable we refer the reader to \cite{SED_Selen2015}. In \cite{SED_Selen2015} the authors establish the following properties:
\begin{enumerate}[label = \arabic*.]%
\item There exists a positive integer $N$ such that the series in \eqref{eqn:equilibrium_distribution_3d_pos} and \eqref{eqn:equilibrium_distribution_3d_neg} converge absolutely for all $m \ge 0, ~ |n| \ge 1$ with $m + |n| > N$ and the series \eqref{eqn:equilibrium_distribution_3d_hor} converges absolutely for all $m \ge N$.
\item For $m + |n| > N$, we have $|x(l,m,n,r)| < u(l)$ and $\sum_{l = 0}^\infty u(l) < \infty$.
\item The series $\sum_{m + |n| > N} p(m,n,r), ~ r = 0,1,\ldots,s - 1$ converges absolutely.
\item $|\al_{l,i}| > |\be_{l,d(i) + j}|$ and $|\be_{l,i}| > |\al_{l + 1,i}|$ with $\al_{0,1} = \rho^{1 + s} < 1$.
\end{enumerate}%
In this proof we make use of the dominated convergence theorem for complex-valued functions.

\subsection{Server 1}%
\label{app_subsec:limiting_conditional_arrival_rate_server_1}%

The limiting conditional arrival rate to server 1 can be determined from
\begin{equation}%
\lim_{n \to \infty} \la_1(n) = \lim_{n \to \infty} \frac{\pi_1(n + 1)/\al_{0,1}^{n + 1}}{\pi_1(n)/\al_{0,1}^{n}} \al_{0,1}. \label{eqn:app_limiting_conditional_arrival_rate_server_1}
\end{equation}%
The marginal distribution for server 1 is given by, where $m = \lfloor \frac{q_2}{s} \rfloor$ and $r = \modu{q_2,s}$,
\begin{align}%
\pi_1(n) &= \sum_{m = 0}^\infty \sum_{r = 0}^{s - 1} p(\min(n,m),m - n,r) \notag \\
&= \sum_{m = 0}^{n - 1} \sum_{r = 0}^{s - 1} p(m,m - n,r) + \sum_{r = 0}^{s - 1} p(n,0,r) + \sum_{m = 1}^\infty \sum_{r = 0}^{s - 1} p(n,m,r). \label{eqn:app_marginal_server_1}
\end{align}%
Furthermore,
\begin{align}%
\lim_{n \to \infty} \frac{\pi_1(n)}{\al_{0,1}^{n}} &= \lim_{n \to \infty} \sum_{m = 0}^{n - 1} \sum_{r = 0}^{s - 1} \frac{p(m,m - n,r)}{\al_{0,1}^{n}} \notag \\
&\quad + \sum_{r = 0}^{s - 1} \lim_{n \to \infty} \frac{p(n,0,r)}{\al_{0,1}^{n}} + \sum_{m = 1}^\infty \sum_{r = 0}^{s - 1} \lim_{n \to \infty} \frac{p(n,m,r)}{\al_{0,1}^{n}}, \label{eqn:app_marginal_server_1_scaled}
\end{align}%
where the interchange of the limit and the series for the third term on the right-hand side of \eqref{eqn:app_marginal_server_1_scaled} is allowed by the dominated convergence theorem, because one can bound $p(n,m,r)$ from above by $p(0,m,r)$ and $\sum_{m = 0}^\infty p(0,m,r) < \infty$ since it is a subseries of $\sum_{m + |n| > N} p(m,n,r)$, which converges absolutely by property 3. Furthermore, we know that $\lim_{m \to \infty} p(m,n,r) = \lim_{m \to \infty} \sum_{l = 0}^\infty x(l,m,n,r)$ which is equal to $\sum_{l = 0}^\infty \lim_{m \to \infty} x(l,m,n,r)$ by the dominated convergence theorem for complex-valued functions in combination with property 2. This allows us to compute the second and third term on the right-hand side of \eqref{eqn:app_marginal_server_1_scaled}. The first term on the right-hand side of \eqref{eqn:app_marginal_server_1_scaled} can be determined as follows
\begin{align}%
&\lim_{n \to \infty} \sum_{m = 0}^{n - 1} \sum_{r = 0}^{s - 1} \frac{p(m,m - n,r)}{\al_{0,1}^{n}} \notag \\
&= C \Bigl(\lim_{n \to \infty} \sum_{l = 0}^\infty \sum_{i = 1}^{(s + 1)^l} \ch_{l,i(s + 1)} \frac{\left( \frac{\al_{l,i}}{\al_{0,1}} \right)^{n} - \left( \frac{\be_{l,i(s + 1)}}{\al_{0,1}} \right)^{n}}{\frac{\al_{l,i}}{\be_{l,i(s + 1)}} - 1} \sum_{r = 0}^{s - 1} \ineg{\al_{l,i},\be_{l,i(s + 1)},r} \notag \\
&\quad+ \lim_{n \to \infty} \sum_{l = 0}^\infty \sum_{i = 1}^{(s + 1)^l} \cv_{l + 1,i(s + 1)} \frac{\left( \frac{\be_{l,i(s + 1)}}{\al_{0,1}} \right)^{n} - \left( \frac{\al_{l + 1,i(s + 1)}}{\al_{0,1}} \right)^{n}}{1 - \frac{\al_{l + 1,i(s + 1)}}{\be_{l,i(s + 1)}}} \sum_{r = 0}^{s - 1}\ineg{\al_{l + 1,i(s + 1)},\be_{l,i(s + 1)},r} \Bigr) \notag \\
&= C \frac{\ch_{0,s + 1}}{\frac{\al_{0,1}}{\be_{0,s + 1}} - 1} \sum_{r = 0}^{s - 1} \ineg{\al_{0,1},\be_{0,s + 1},r}.
\end{align}%
Interchange of the limit and series is again allowed here since one can bound the absolute value of the summands from above by $v(l)$ for sufficiently large $n$ and $\sum_{l = 0}^\infty v(l) < \infty$. One can finally establish that
\begin{align}%
\lim_{n \to \infty} \frac{\pi_1(n)}{\al_{0,1}^{n}} = C \Bigl(& \frac{\ch_{0,s + 1}}{\frac{\al_{0,1}}{\be_{0,s + 1}} - 1} \sum_{r = 0}^{s - 1} \ineg{\al_{0,1},\be_{0,s + 1},r} \notag \\
&+ \sum_{r = 0}^{s - 1} h_{0,1}(r) + \sum_{j = 1}^s \frac{\ch_{0,j} \be_{0,j}}{1 - \be_{0,j}} \sum_{r = 0}^{s - 1} \ipos{\al_{0,1},\be_{0,j},r} \Bigr). \label{eqn:app_marginal_server_1_scaled_computed}
\end{align}%
Thus, by \eqref{eqn:app_limiting_conditional_arrival_rate_server_1} and \eqref{eqn:app_marginal_server_1_scaled_computed}, $\lim_{n \to \infty} \la_1(n) = \al_{0,1} = \rho^{1 + s}$.

\subsection{Server 2}%
\label{app_subsec:limiting_conditional_arrival_rate_server_2}%

The limiting conditional arrival rate to server 2 can be determined from
\begin{equation}%
\lim_{n \to \infty} \la_2(sn + r) = \begin{cases}%
\lim_{n \to \infty} s \frac{\pi_2(sn + r + 1)/\al_{0,1}^{n}}{\pi_2(sn + r)/\al_{0,1}^{n}}, & r = 0,1,\ldots,s - 2, \\
\lim_{n \to \infty} s \frac{\pi_2(sn + r + 1)/\al_{0,1}^{n + 1}}{\pi_2(sn + r)/\al_{0,1}^{n}} \al_{0,1}, & r = s - 1.
\end{cases}%
\label{eqn:app_limiting_conditional_arrival_rate_server_2}
\end{equation}%
The marginal distribution for server 2 is given by, for $r = 0,1,\ldots,s - 1$,
\begin{align}%
\pi_2(sn + r) &= \sum_{q_1 = 0}^\infty p(\min(q_1,\lfloor \frac{sn + r}{s} \rfloor),\lfloor \frac{sn + r}{s} \rfloor - q_1,\modu{sn + r,s}) \notag \\
&= \sum_{q_1 = 0}^\infty p(\min(q_1,n),n - q_1,r) \notag \\
&= \sum_{q_1 = 0}^{n - 1} p(q_1,n - q_1,r) + p(n,0,r) + \sum_{q_1 = 1}^\infty p(n,-q_1,r). \label{eqn:app_marginal_server_2}
\end{align}%
For $\pi_2(sn + r + 1), ~ r = s - 1$ we should replace $n$ by $n + 1$ and $r$ by 0 in \eqref{eqn:app_marginal_server_2}. Furthermore, for $r = 0,1,\ldots,s - 1$,
\begin{equation}%
\lim_{n \to \infty} \frac{\pi_2(sn + r)}{\al_{0,1}^{n}} = \lim_{n \to \infty} \sum_{q_1 = 0}^{n - 1} \frac{p(q_1,n - q_1,r)}{\al_{0,1}^n} + \lim_{n \to \infty} \frac{p(n,0,r)}{\al_{0,1}^n} + \lim_{n \to \infty} \sum_{q_1 = 1}^\infty \frac{p(n,- q_1,r)}{\al_{0,1}^n}. \label{eqn:app_marginal_server_2_scaled}
\end{equation}%
Using identical arguments as for the limiting conditional arrival rate for server 1, we establish
\begin{equation}%
\lim_{n \to \infty} \frac{\pi_2(sn + r)}{\al_{0,1}^{n}} = A(r), \quad r = 0,1,\ldots,s - 1, \label{eqn:app_marginal_server_2_scaled_computed}
\end{equation}%
where $A(r)$ is given in \eqref{eqn:A(r)}. Finally, combining \eqref{eqn:app_limiting_conditional_arrival_rate_server_2} and \eqref{eqn:app_marginal_server_2_scaled_computed} proves \eqref{eqn:limiting_conditional_arrival_rate_server_2}.


\section{Proof of Proposition~\ref{prop:conditional_arrival_rates_limiting_regimes}, point 2}%
\label{app_sec:proof_limiting_regimes}%

By letting $\rho \downarrow 0$ in \eqref{eqn:conditional_arrival_rate_server_1_n_0}-\eqref{eqn:conditional_arrival_rate_server_1_n_2} and $\la_1(n) \approx \rho^{1 + s}, ~ n \ge 3$ we immediately find that $\la_1(n) \downarrow 0, ~ n \in \Nat_0$.

We note that in \eqref{eqn:conditional_arrival_rate_server_2_n_s-1_until_2s-1} the factors on the right-hand side in front of $\la_2^\textup{lim}(n-s)$ go to a constant for $\rho \downarrow 0$. So, what remains is that we establish that $\lim_{\rho \downarrow 0} \la_2^\textup{lim}(r) = 0, ~ r = 0,1,\ldots,s-1$. This part of the proof relies heavily on the asymptotic results of \cite{SED_Selen2015}. We denote $\al = \rho^{1 + s}$ and investigate for $r = 0,1,\ldots,s - 1$,
\begin{equation}%
\frac{A(\al,r)}{\al^{r/s}} = \sum_{i = 1}^s \ch_i \frac{\be_i/\al}{1 - \be_i/\al} \unit_i \bigl( \frac{\be_i}{\al} \bigr)^{r/s} + \al^{1/s} \frac{h(r)}{\al^{(r + 1)/s}} + \al^{1 - r/s} \frac{\be_{s + 1}/\al}{1 - \be_{s + 1}} \ineg{\al,\be_{s + 1},r},
\end{equation}%
where we used that $\ipos{\al,\be_i,r} = \unit_i \be_i^{r/s}$ with $\unit_i$ the $i$-th unit root of $\unit^s = 1$, which is established in \cite[Lemma~5.6]{SED_Selen2015}. Now,
\begin{equation}%
\lim_{\al \downarrow 0} \frac{A(\al,r)}{\al^{r/s}} = c(r), \label{eqn:app_limiting_scaled_A(r)}
\end{equation}%
where $c(r)$ is some constant. In the following we denote $c_i$ as some constant that can be a function of $r$. Equation \eqref{eqn:app_limiting_scaled_A(r)} follows from the fact that for $\al \downarrow 0$ we have that $\be_i/\al \to c_1 < 1, ~ i = 1,2,\ldots,s$, $\be_{s + 1}/\al \to c_2$ \cite[Lemma~5.15(i)(a) and (i)(c)]{SED_Selen2015}; $h(r)/\al^{(r + 1)/s} \to c_3(r)$ \cite[Appendix~B, part (c)]{SED_Selen2015}; $\ineg{\al,\be_{s + 1},r} \to c_4(r)$ \cite[Lemma~5.15(i)(d)]{SED_Selen2015}; $\be_{s + 1} \to 0$ \cite[Corollary~5.14]{SED_Selen2015}; and $\sum_{i = 1}^s \ch_i \unit_i \to c_5$ ($\al \downarrow 0$ in (5.46) of \cite{SED_Selen2015}).

Finally, for $r = 0,1,\ldots,s - 2$,
\begin{equation}%
\lim_{\al \downarrow 0} \frac{A(\al,r + 1)}{A(\al,r)} = \lim_{\al \downarrow 0} \frac{\al^{(r + 1)/s}}{\al^{r/s}} \frac{A(\al,r + 1)/\al^{(r + 1)/s}}{A(\al,r)/\al^{r/s}}
= \lim_{\al \downarrow 0} \al^{1/s} \frac{c(r + 1)}{c(r)} = 0
\end{equation}%
and for $r = s - 1$,
\begin{equation}%
\lim_{\al \downarrow 0} \al \frac{A(\al,0)}{A(\al,s - 1)} = \lim_{\al \downarrow 0} \frac{\al}{\al^{(s - 1)/s}} \frac{A(\al,0)}{A(\al,r)/\al^{(s - 1)/s}} = \lim_{\al \downarrow 0} \al^{1/s} \frac{c(0)}{c(s - 1)} = 0.
\end{equation}%
This concludes the proof.


\bibliographystyle{plain}%
\bibliography{BibliographyServerFarmsGJSQNew}%

\begin{thebibliography}{10}

\bibitem{HeterogeneousPSFarms_LoadBalancing_Altman2011}
E.~Altman, U.~Ayesta, and B.~Prabhu.
\newblock {Load balancing in processor sharing systems}.
\newblock {\em Telecommunication Systems}, 47(1-2):35--48, 2011.

\bibitem{Banawan1989}
S.A. Banawan and J.~Zahorjan.
\newblock {Load sharing in heterogeneous queueing systems}.
\newblock In {\em IEEE INFOCOM '89}, pages 731--739, 1989.

\bibitem{Cardellini2002}
V.~Cardellini, E.~Casalicchio, M.~Colajanni, and P.S. Yu.
\newblock {The state of the art in locally distributed web-server systems}.
\newblock {\em ACM Computing Surveys}, 34(2):263--311, 2002.

\bibitem{Foschini1978}
G.J. Foschini and J.~Salz.
\newblock {A basic dynamic routing problem and diffusion}.
\newblock {\em IEEE Transactions on Communications}, 26(3):320--327, 1978.

\bibitem{JSQ_WebServerFarms_Gupta2007}
V.~Gupta, M.~Harchol-Balter, K.~Sigman, and W.~Whitt.
\newblock {Analysis of join-the-shortest-queue routing for web server farms}.
\newblock {\em Performance Evaluation}, 64(9):1062--1081, 2007.

\bibitem{Harchol-Balter2013}
M.~Harchol-Balter.
\newblock {\em {Performance Modeling and Design of Computer Systems: Queueing
  Theory in Action}}.
\newblock Cambridge University Press, 2013.

\bibitem{HB2003}
M.~Harchol-Balter, B.~Schroeder, N.~Bansal, and M.~Agrawal.
\newblock {Size-based scheduling to improve web performance}.
\newblock {\em ACM Transactions on Computer Systems}, 21(2):207--233, 2003.

\bibitem{HeterogeneousServerFarms_SizeStateAwareRouting_Hyytia2012}
E.~Hyyti\"{a}, A.~Penttinen, and S.~Aalto.
\newblock {Size- and state-aware dispatching problem with queue-specific job
  sizes}.
\newblock {\em European Journal of Operational Research}, 217(2):357--370,
  2012.

\bibitem{Laws1992}
C.N. Laws.
\newblock {Resource pooling in queueing networks with dynamic routing}.
\newblock {\em Advances in Applied Probability}, 24(3):699--726, 1992.

\bibitem{SLA}
Z.~Liu, M.S. Squillante, and J.L. Wolf.
\newblock On maximizing service-level-agreement profits.
\newblock In {\em ACM conference on Electronic Commerce}, pages 213--223. ACM,
  2001.

\bibitem{Lui1995}
J.C.S. Lui, R.R. Muntz, and D.~Towsley.
\newblock {Bounding the mean response time of the minimum expected delay
  routing policy: An algorithmic approach}.
\newblock {\em IEEE Transactions on Computers}, 44(12):1371--1382, 1995.

\bibitem{Nelson1989}
R.D. Nelson and T.K. Philips.
\newblock {An approximation to the response time for shortest queue routing}.
\newblock In {\em SIGMETRICS '89}, pages 181--189, 1989.

\bibitem{Ortiz}
M.C. Ortiz.
\newblock {Building the dynamic data center}.
\newblock {\em Dell Power Solutions}, 03:50--53, 2010.

\bibitem{SED_Selen2015}
J.~Selen, I.J.B.F. Adan, S.~Kapodistria, and J.S.H. van Leeuwaarden.
\newblock Steady-state analysis of shortest expected delay routing.
\newblock arXiv:1509.03535v2, 2015.

\bibitem{Whitt1986}
W.~Whitt.
\newblock {Deciding which queue to join: Some counterexamples}.
\newblock {\em Operations Research}, 34(1):55--62, 1986.

\end{thebibliography}

\end{document}